\documentclass{jonasart}

\usepackage{ytableau}

\newcommand{\CC}{\mathbb{C}}
\newcommand{\cdotlr}{\mathop{\hspace{-2pt}\underset{\text{LR}}{\cdot}\hspace{-2pt}}}
\newcommand{\cdotrl}{\mathop{\hspace{-2pt}\underset{\text{RL}}{\cdot}\hspace{-2pt}}}
\newcommand{\dy}{ DY_c(\mathfrak{gl}_{N})}
\newcommand{\Ec}{\mathcal{E}}
\newcommand{\fand}{\quad\text{and}\quad}
\newcommand{\Fand}{\qquad\text{and}\qquad}

\newcommand{\gl}{\mathfrak{gl}}
\newcommand{\R}{\overline{R}}
\newcommand{\Sc}{\mathcal{S}}

\newcommand{\ts}{\hspace{1pt}}
\newcommand{\vac}{\mathrm{\boldsymbol{1}}}
\newcommand{\Vc}{\mathcal{V}^{c} (\mathfrak{gl}_N)}
\newcommand{\Vn}{\mathcal{V}^{crit} (\mathfrak{gl}_N)}
\newcommand{\Vndva}{\mathcal{V}^{crit} (\mathfrak{gl}_2)}
\newcommand{\ZZ}{\mathbb{Z}}

\DeclareMathOperator{\diag}{diag}

\DeclareMathOperator{\ndo}{End}
\DeclareMathOperator{\om}{Hom}
\DeclareMathOperator{\rez}{Res}
\DeclareMathOperator{\tr}{tr}

\title[Evaluation-type deformed modules over the quantum affine vertex algebras of type $A$]{Evaluation-type deformed modules over the quantum \\ affine vertex algebras of type $A$}

\authors{Lucia Bagnoli and Slaven Ko\v{z}i\'{c}}

\authorinfo[L. Bagnoli]{Sapienza University of Rome and INFN Rome Division, Italy}{lucia.bagnoli@uniroma1.it}

\authorinfo[S. Ko\v{z}i\'{c}]{University of Zagreb, Croatia}{slaven.kozic@math.hr}

\abstract{%
    Let $\mathcal{V}^c(\mathfrak{gl}_N)$ be Etingof--Kazhdan's quantum affine vertex algebra associated with the trigonometric $R$-matrix. We establish a connection between suitably generalized deformed $\phi$-coordinated $\mathcal{V}^c(\mathfrak{gl}_N)$-modules and the representations of quantized enveloping algebra $U_h(\mathfrak{gl}_N)$ and reflection equation algebra $\mathcal{O}_h(Mat_N)$. As an application, we demonstrate how the elements of the center of $\mathcal{V}^c(\mathfrak{gl}_N)$ at the critical level $c=-N$ give rise to the $q$-analogues of quantum immanants for $U_h(\mathfrak{gl}_N)$, which were recently found by Jing, Liu and Molev. Finally, we derive the analogues of these results for the quantum affine vertex algebra associated with the normalized Yang $R$-matrix.
}



\acknowledgments{%
	The authors would like to thank the   referees for the   comments     which helped them to improve the manuscript. L. B. is a~member of Gruppo Nazionale per le Strutture Algebriche, Geometriche e le loro Applicazioni (GNSAGA) of the Istituto Nazionale di Alta Matematica (INdAM) and part of the project MMNLP (Mathematical Methods in Non Linear Physics) of INFN. L. B. was partially supported by the project Representation Theory and Applications, Bando Ateneo 2023 of Sapienza University of Rome. S. K. is partially supported by the Croatian Science
	Foundation under the project IP-2025-02-4720 and by the project ``Implementation of cutting-edge research and its application as part of the Scientific Center of Excellence for Quantum and Complex Systems, and Representations of Lie Algebras'', Grant No. PK.1.1.10.0004, co-financed by the European Union through the European Regional Development Fund - Competitiveness and Cohesion Programme 2021--2027. This paper was supported by the European Union – NextGenerationEU through the National Recovery and Resilience Plan 2021-2026. Institutional grant of University of Zagreb Faculty of Science Strengthening scientific production, international presence and social impact of mathematical research (IK IA 1.1.3. Impact4Math).
}

\DOI{https://doi.org/10.46298/jonas.18022}

\begin{document}
	\section{Introduction}

	Let $U_h (\mathfrak{gl}_N)$ be the \emph{quantized enveloping algebra} over the commutative ring $\CC[[h]] $ of formal power series in parameter $h$. Its presentation can be given using the $R$-matrix
	\begin{equation}\label{R}
	R=
	e^{h/2}\sum_{i} e_{ii}\otimes e_{ii}
	+\sum_{i\neq j} e_{ii}\otimes e_{jj}
	+\left(e^{h/2}-e^{-h/2}\right)\sum_{i<j}e_{ij}\otimes e_{ji},
	\end{equation} where $e_{ij}\in\ndo\CC^N$ are the matrix units. The algebra $U_h (\mathfrak{gl}_N)$ is generated by the elements
	$l_{ij}^+$ and $l_{ji}^-$, where $1\leqslant i\leqslant j\leqslant N$, subject to the defining relations
	\begin{gather}
		l_{ii}^+ \ts l_{ii}^- = l_{ii}^-\ts l_{ii}^+\quad\text{for }i=1,\ldots,N,\nonumber\\
		R\ts L_1^\pm\ts L_2^\pm = L_2^\pm\ts L_1^\pm\ts R\Fand
		R\ts L_1^+\ts L_2^- = L_2^-\ts L_1^+\ts R.\label{rlls}
	\end{gather}
	Note that the last two families of relations are expressed over the tensor product algebra $\ndo\CC^N\otimes\ndo\CC^N\otimes U_h (\mathfrak{gl}_N)$ using the matrices
	\[
		L^\pm_1 =\sum_{i,j=1}^N e_{ij}\otimes 1\otimes l_{ij}^\pm\Fand
		L^\pm_2 =\sum_{i,j=1}^N 1\otimes e_{ij}\otimes l_{ij}^\pm
	\]
	with $l_{ij}^+ =l_{ji}^- = 0$ for $1\leqslant j< i\leqslant N$.

	The \emph{reflection equation algebra} $\mathcal{O}_h (Mat_N)$
	over the ring $\CC[[h]] $ can be also presented in terms of the $R$-matrix~\eqref{R}. It is defined as the algebra generated by the elements $\ell_{ij}$, where $i,j=1,\ldots,N$, subject to the defining relations, expressed over the tensor product algebra $\ndo\CC^N\otimes\ndo\CC^N\otimes \mathcal{O}_h (Mat_N)$,
	\begin{equation}\label{def_r efl}
	R \ts \mathcal{L}_1\ts R_{21}\ts \mathcal{L}_2
	=\mathcal{L}_2\ts R \ts \mathcal{L}_1\ts R_{21},
	\qquad\text{where}\qquad
	\mathcal{L}=\sum_{i,j=1}^N e_{ij}\otimes \ell_{ij}.
	\end{equation}
	Note that in~\eqref{def_r efl}, as with~\eqref{rlls} above, we use the standard tensor notation so that
	\[
		\mathcal{L}_1 =\sum_{i,j=1}^N e_{ij}\otimes1\otimes \ell_{ij}
		\fand
		\mathcal{L}_2 =\sum_{i,j=1 }^N 1\otimes e_{ij}\otimes \ell_{ij}
	\]
	and $R_{21}=PRP$, where $P$ is the permutation operator
	\begin{equation}\label{permutation}
	P=\sum_{i,j=1}^{N} e_{ij}\otimes e_{ji}.
	\end{equation}
	Throughout the paper, we assume that the algebras $\mathcal{O}_h (Mat_N)$ and $U_h (\mathfrak{gl}_N)$ are completed with respect to the $h$-adic topology. For more information on quantized enveloping algebras and reflection equation algebras see, e.g.,~\cite{GPyS,JW,KS}.

	Let $c\in\CC$. Denote by $\Vc$ the Etingof--Kazhdan \emph{ quantum affine vertex algebra at the level $c$}~\cite{EK} associated with the $R$-matrix
	\begin{equation}\label{Reu}
	R(e^u) = f(u)\left(R + \left(e^{h/2}-e^{-h/2}\right)e^u\left(1-e^u\right)^{-1} P\right)\in
	\ndo\CC^N \otimes \ndo\CC^N ((u))[[h]],
	\end{equation} where $f(u)$ is a~certain normalization series\footnote{For more information, see~\cite[Prop. 1.2]{EK4} and~\cite[Prop. 4.7]{FR}. The normalization term $f(u)$ in $1+h\CC((u))[[h]]$ coincides, up to a~nonzero scalar, with $ (1-e^u) (1-e^{u-h })^{-1}g(u)$ with $g(u)$ given by~\cite[Prop.~2.1]{KM}.} while $R$ and $P$ are given by~\eqref{R} and~\eqref{permutation} respectively.

	The goal of this paper is to establish a~connection between the quantum vertex algebra $\Vc$ and the representations of the quantized enveloping algebra $U_h (\mathfrak{gl}_N)$ and the reflection equation algebra $\mathcal{O}_h (Mat_N)$. Our main tool is Li's theory of \emph{$\phi$-coordinated modules}~\cite{Liphi}, which was introduced to associate quantum vertex algebras with quantum affine algebras; see also~\cite{JKLT, Kong, Kong2} and the references therein. Namely, the quantum vertex algebra $\Vc$ is essentially an additive structure, in the sense that its braiding $\Sc$ (defined by~\eqref{braiding_r eqs2} below) satisfies the \emph{additive} quantum Yang--Baxter equation,
	\[
		\Sc_{12}(z_1)\ts \Sc_{13}(z_1 +z_2)\ts \Sc_{23}(z_2)
		=
		\Sc_{23}(z_2)\ts \Sc_{13}(z_1 +z_2)\ts\Sc_{12}(z_1).
	\]
	On the other hand, the vertex operators which we use to establish the aforementioned connection induce the braiding $\sigma$ (defined by~\eqref{sigma_m ap} below) which satisfies the \emph{multiplicative} quantum Yang--Baxter equation,
	\[
		\sigma_{12}(z_1)\ts \sigma_{13}(z_1 z_2)\ts \sigma_{23}(z_2)
		=
		\sigma_{23}(z_2)\ts \sigma_{13}(z_1 z_2)\ts\sigma_{12}(z_1).
	\]
	This structural gap is bridged by the use of the $\phi$-coordinated modules theory.

	Another challenge in connecting the representations of $U_h (\mathfrak{gl}_N)$ and $\mathcal{O}_h (Mat_N)$ with $\Vc$ is the fact that the corresponding braidings $\Sc$ and $\sigma$ are not related in a~usual way, i.e., by an associate of the one-dimensional additive formal group; cf.~\cite[Sect. 2]{Liphi}. Thus, we cannot directly apply the theory of $\phi$-coordinated modules. Instead, to set up the right framework, we generalize the notion of \emph{ deformed} $\phi$-coordinated module from our recent joint work with Jing~\cite{BJK}; see also~\cite{BK0,BK}. While the ordinary notion of ($\phi$-coordinated) module features the weak associativity property for the underlying module map, this is no longer the case for the deformed ($\phi$-coordinated) modules. However, this nonassociativity in the deformed ($\phi$-coordinated) module setting occurs in a~controlled way, governed by a~certain map $\rho$, which provides a~partial link between $\Sc$ and $\sigma$.

	The main results of the manuscript are Corollary~\ref{kor_4 4} and Theorem~\ref{main_t thm} which, under certain conditions, establish a~connection between the representations of $U_h (\mathfrak{gl}_N)$ and $\mathcal{O}_h (Mat_N)$ and the deformed $\phi$-coordinated $\Vc$-modules. Next, we investigate the images of certain families of central elements of $\Vc$ at the critical level $c=-N$, found in~\cite{BJK, KM}, under the module map given by Corollary~\ref{kor_4 4}. Note that the classical limit $h\to 0$ of $\Vc$ is the level $c$ universal affine vertex algebra $V^c (\mathfrak{gl}_N)$~\cite{FZ,Lian}, so that these families are quantum analogues of the elements of the Feigin--Frenkel center~\cite{FF}. We show that their images produce $q$-analogues of Okounkov's quantum immanants~\cite{Ok}, which were recently found by Jing, Liu and Molev~\cite{JLM}.

	At the end, we use the construction from our recent work~\cite{BK0} to partially extend these results to the case of Etingof--Kazhdan's quantum affine vertex algebra associated with the suitably normalized Yang $R$-matrix, which is again denoted by $\Vc $; cf. Section~\ref{section_0 6}. This yields a~connection between representations of the universal enveloping algebra $U (\mathfrak{gl}_N)$ and deformed $\Vc $-modules. Also, we show that certain elements of the center of $\Vc $ at the critical level $c=-N$, which were found in~\cite{JKMY}, give rise to Okounkov's quantum immanants.

	\section{Preliminaries}

	In this section, we recall the trigonometric $R$-matrix, which governs the RTT\--re\-al\-i\-za\-tion of the quantum affine algebra in type $A$; see~\cite{FRT,J,PS,RS} for more information. Next, we recollect Etingof--Kazhdan's construction~\cite{EK} of the quantum affine vertex algebra associated with this $R$-matrix.

	\subsection{Trigonometric $R$-matrix}

	Let
	\begin{equation}\label{rbar_1}
	\R(z)=
	R+\frac{(e^{h/2 } -e^{-h/2}) z}{1-z}P \in \ndo\CC^N\otimes\ndo\CC^N [[z,h]],
	\end{equation} where the operators $R$ and $P$ are given by~\eqref{R} and~\eqref{permutation}, respectively. Denote by
	\[
		f_q (z)=1+\sum_{r=1}^\infty f_{q,r} \frac{z^r}{(1-z)^{r}}\in\CC(q)[[z]]
	\]
	a~unique formal power series such that all $f_{q,r}(q-1)^{-r}$ are regular at $q=1$, satisfying
	\[
		f_q (zq^{2N})=f_q (z) (1-zq^2)(1-zq^{2N-2}) (1-z)^{-1}(1-zq^{2N})^{-1};
	\]
	see~\cite{EK4,FR,KM} for more information on $f_q (z)$. By setting $q=e^{h/2},$ we obtain
	\begin{equation}\label{rbar_2}
	f_{e^{h/2}}(z)=1+\sum_{r=1}^\infty f_{r} \frac{z^r}{(1-z)^{r}}\in\CC[[h,z]],\quad\text{where}\quad f_r =f_{q,r}\big|_{q=e^{h/2}}.
	\end{equation}
	Finally, define the normalized $R$-matrix by
	\begin{equation}\label{R_m ult}
	R(z)= \frac{e^{-h/2}(1-z) f_{e^{h/2}}(z)}{1-ze^{-h}} \ts \R(z) \in \ndo\CC^N\otimes\ndo\CC^N [[z,h]].
	\end{equation}
	The $R$-matrix $R(z)$ can be also regarded as an element of $ \ndo\CC^N\otimes\ndo\CC^N (z)[[h]] $, i.e., as a~rational function with respect to $z$. It satisfies the \emph{quantum Yang--Baxter equation},
	\begin{equation}
	R_{12}(z_1)\ts R_{13}(z_1 z_2)\ts R_{23}(z_2)
	= R_{23}(z_2)\ts R_{13}(z_1 z_2)\ts R_{12}(z_1)
	\end{equation} and the \emph{crossing symmetry properties,}
	\begin{equation}\label{csym_m}
	R(ze^{ Nh})^{t_1} \ts D_1\ts (R(z)^{-1})^{t_1}=D_1\fand (R(z)^{-1})^{t_2} \ts D_2 \ts R(z e^{ Nh})^{t_2} = D_2,
	\end{equation} where
	$t_i$ denotes the transposition $e_{rs}\mapsto e_{sr}$ applied on the $i$-th tensor factor and
	\begin{equation}\label{diagonal}
	D=\diag\left(e^{(N-1)h/2},e^{(N-3)h/2},\ldots, e^{-(N-1)h/2}\right) \in\ndo\CC^N [[h]].
	\end{equation}

	Consider the unique embedding
	$\CC_* (u)\hookrightarrow\CC((u)),$ where $\CC_* (u)$ denotes the localization of the ring of formal Taylor series $\CC[[u]]$ at $\CC[u]\setminus\left\{0\right\}.$
	It naturally extends to the embedding
	\begin{equation}\label{iotau_r ev}
	\iota_u \colon \CC_* (u)[[h]]\hookrightarrow\CC((u)) [[h]].
	\end{equation}
	By setting $z=e^u$ in~\eqref{R_m ult} and then applying
	$\iota_u$ to the resulting expression, one obtains, up to a~nonzero scalar, the $R$-matrix~\eqref{Reu}.

	Let $u=(u_1,\ldots,u_n)$ and $v=(v_1,\ldots,v_m)$ be two families of variables, $z$ be a~single variable, and $a\in\CC$. Define the following formal power series with coefficients in
	$(\ndo\CC^N)^{\otimes n} \otimes (\ndo\CC^N)^{\otimes m}$,
	\begin{equation}\label{Rnm1}
	R_{nm}^{12}(e^{z+u -v +ah})=
	\prod_{r=1,\dots,n }^{\longrightarrow} \prod_{s=n+1,\dots,n+m }^{\longleftarrow}
	R_{rs}(e^{z+ u_r -v_{s-n} + ah}),
	\end{equation} where the arrows indicate the order of the factors and the superscript $1$ (resp. $2$) corresponds to the tensor factors $1,\ldots,n$ (resp. $n+1,\ldots,n+m$). The same expression with variable $z$ omitted is denoted by $R_{nm}^{12}(e^{u -v +ah})$. We shall apply the same notation convention to the $R$-matrices~\eqref{R} and~\eqref{R_m ult}, so that we have
	\begin{align}
		&R_{nm}^{12} =
		\prod_{r=1,\dots,n }^{\longrightarrow} \prod_{s=n+1,\dots,n+m }^{\longleftarrow}
		R_{rs},\label{Rnm}\\
		&R_{nm}^{12}(ze^{ u -v +ah})=
		\prod_{r=1,\dots,n }^{\longrightarrow} \prod_{s=n+1,\dots,n+m }^{\longleftarrow}
		R_{rs}(ze^{ u_r -v_{s-n} +ah}).
	\end{align}
	Furthermore, for $x=(x_1,\ldots,x_n)$ and $y=(y_1,\ldots,y_m)$, we write
	\begin{equation}\label{Rnm4}
	R_{nm}^{12}(x/y)=
	\prod_{r=1,\dots,n }^{\longrightarrow} \prod_{s=n+1,\dots,n+m }^{\longleftarrow}
	R_{rs}(x_r /y_{s-n}).
	\end{equation}

	By swapping the superscripts $1$ and $2$ in~\eqref{Rnm1}--\eqref{Rnm4}, we shall indicate that the factors on the right-hand side are conjugated by the permutation operator $P$ and that the product is taken in the opposite order, e.g., we have
	\[
		R_{nm}^{21} =
		\prod_{r=1,\dots,n }^{\longleftarrow} \prod_{s=n+1,\dots,n+m }^{\longrightarrow}
		P_{rs}\ts R_{rs}\ts P_{rs} =
		\prod_{r=1,\dots,n }^{\longleftarrow} \prod_{s=n+1,\dots,n+m }^{\longrightarrow}
		R_{sr}.
	\]
	We shall also need certain products where the order of factors differs from the one in~\eqref{Rnm}--\eqref{Rnm4}. They will be denoted by the underlined symbols. We write
	\begin{align}
		&\underline{R}_{nm}^{12} =
		\prod_{r=1,\dots,n }^{\longrightarrow} \prod_{s=n+1,\dots,n+m }^{\longrightarrow}
		R_{rs}\Fand
		\underline{R}_{nm}^{21} =
		\prod_{r=1,\dots,n }^{\longrightarrow} \prod_{s=n+1,\dots,n+m }^{\longrightarrow}
		R_{sr}.
	\end{align}
	By swapping the superscripts $1$ and $2$ in~\eqref{Rnm1}--\eqref{Rnm4} and underlining the corresponding symbol, we shall indicate that the factors on the right-hand side are conjugated by the permutation operator $P$, e.g., we have
	\begin{equation}
	\underline{R}_{nm}^{21}(ze^{ u -v +ah}) =
	\prod_{r=1,\dots,n }^{\longrightarrow} \prod_{s=n+1,\dots,n+m }^{\longleftarrow}
	R_{sr}(ze^{ u_r -v_{s-n} +ah}).
	\end{equation}
	Finally, we set $R_{[1]}=1$ and
	\begin{equation}\label{rev3_R n}
	R_{[n]}=
	\prod_{r=1,\dots,n-1 }^{\longrightarrow}\prod_{s=r+1,\dots,n }^{\longrightarrow} R_{sr} \quad\text{for }n>1.
	\end{equation}

	Denote by $R(z)^\prime$ the inverse of the $R$-matrix~\eqref{R_m ult} in
	$ \left(\ndo\CC^N\right)^{\text{op}} \otimes \ndo\CC^N [[z,h]]
	$. The explicit expression for $R(z)^\prime$ can be derived from the crossing symmetry properties~\eqref{csym_m}. Moreover, as $R(z)\left|_{z=0} =e^{-h/2}R\right.$, by setting $z=0$ in~\eqref{csym_m}, we see that the $R$-matrix $R$, given by~\eqref{R}, is also invertible in
	$\left(\ndo\CC^N\right)^{\text{op}} \otimes \ndo\CC^N [[ h]] $. We denote its inverse by $R^\prime$. Finally, we extend this notation to the products of $R$-matrices~\eqref{Rnm1}--\eqref{Rnm}
	and write, e.g., $\left(R_{nm}^{12}\right)^\prime$, thus indicating that the inverse is taken with respect to the multiplication in
	$\left(\left(\ndo\CC^N\right)^{\text{op}} \right)^{\otimes n}\otimes\left(\ndo\CC^N\right)^{\otimes m}$.

\newpage
	\subsection{Quantum affine vertex algebra \texorpdfstring{$\Vc$}{Vc(glN)}}

	Following~\cite{EK4}, we introduce a~certain algebra which plays a~key role in the Etingof--Kazhdan construction; see also~\cite{FRT,RS}. Let $U(R) $
	be an associative algebra with unit $\vac$ over the ring $\CC[[h]]$ generated by the elements $t_{ij}^{(-r)}$, $i,j=1,\ldots,N$ and $r=1,2,\ldots.$
	Its defining relations are given by
	\begin{equation}
	R(e^{u-v})\ts T^+_{1} (u)\ts T^+_2 (v)= T^+_2 (v)\ts T^+_{1} (u)\ts R(e^{u-v}),
	\end{equation} where the formal power series $T^+ (u) \in\ndo\CC^N \otimes U(R)[[u]]$ is defined by
	\begin{equation}
	T^+ (u) =\sum_{i,j=1}^N e_{ij}\otimes t^+_{ij} (u)
	\qquad\text{with}\qquad t^+_{ij} (u)=\delta_{ij}-h\sum_{r=1}^{\infty}t_{ij}^{(-r)}u^{r-1}
	\end{equation} and the $R$-matrix $R(e^u)$ is defined by~\eqref{Reu}.

	Suppose $V$ is a~topologically free $\CC[[h]]$-module, i.e., $V=V_0 [[h]]$ for some complex vector space $V_0$. We denote by $V((z))_h$ (resp. $V[z^{-1}]_h$) the $h$-adic completion of $V((z))$ (resp. $V[z^{-1}]$). In other words, we have
	$V((z))_h =V_0 ((z))[[h]]$ (resp. $V[z^{-1}]_h = V_0 [z^{-1}][[h]]$). From now on, all tensor products are understood as $h$-adically completed, e.g., we have $V\otimes V =V_0\otimes V_0 [[h]]$.

	Let $\Vc $ be the $h$-adic completion of $U(R) $. The complex parameter $c $ in the superscript determines the action of the operator series $T^- (u)$
	given in the next lemma, which goes back to~\cite[Lemma 2.1]{EK}.

	\begin{lemma}
		For any $c\in\CC$ there exists a~unique invertible operator series
		\[
			T^- (u)\in\ndo\CC^N \otimes \om(\Vc,\Vc ((u))_h)
		\]
		such that for all $n\geqslant 0$ we have
	\begin{equation}
		T^-_{0} (u)\ts T_{1}^+ (v_1)\ldots T_{n}^+ (v_n)\vac=
		R^{01}_{1n}(e^{u-v+hc/2})^{-1}
		\ts T_{1}^+ (v_1)\ldots T_{n }^+ (v_n) \ts R^{01}_{1n}(e^{u-v-hc/2})\vac.
		\end{equation}
	\end{lemma}

	From now on, we regard $T^+ (u)$ as operator series over $\Vc $, such that its action is given by the algebra multiplication. For any $n$ and the variables $u=(u_1,\ldots,u_n)$ define
	\begin{equation}\label{tplusovi_t rig}
	T_{[n]}^{\pm} (u)=T_1^\pm(u_1)\ldots T_n^\pm (u_n)
	\fand
	T_{[n]}^{\pm} (z|u)=T_1^\pm(z+u_1)\ldots T_n^\pm (z+u_n).
	\end{equation}
	The operator series~\eqref{tplusovi_t rig} satisfy the following RTT-relations~\cite[Sect. 2.1.2]{EK}.
	\begin{proposition}
		For any positive integers, $n$ and $m$, and for any families of variables, $u=(u_1,\ldots,u_n)$ and $v=(v_1,\ldots,v_m)$, the following identities hold:
		\begin{gather}
			R_{nm}^{12}(e^{z_1 -z_2 +u-v})T_{[n]}^{\pm 13}(z_1 |u)T_{[m]}^{\pm 23}(z_2 |v)
			=\,T_{[m]}^{\pm 23}(z_2 |v)T_{[n]}^{\pm 13}(z_1 |u)R_{nm}^{12}(e^{z_1 -z_2 +u-v}),\nonumber\\
			R_{nm}^{\ts 12}(e^{z_1 -z_2 +u-v+hc/2})T_{[n]}^{-13}(z_1 |u)T_{[m]}^{+23}(z_2 |v)
			=\,T_{[m]}^{+23}(z_2 |v)T_{[n]}^{-13}(z_1 |u)R_{nm}^{\ts 12}(e^{z_1 -z_2 +u-v-hc/2}).
			\label{RTT3}
		\end{gather}
	\end{proposition}
	We now recall Etingof--Kazhdan's construction~\cite[Thm. 2.3]{EK}. For a~precise definition of the notion of ($h$-adic) quantum vertex algebra see~\cite[Sect. 1.4.1]{EK} and~\cite[Def. 2.4]{Li}.

	\begin{theorem}\label{EK:qva}
		Let $c\in\CC.$ There exists a~unique structure of quantum vertex algebra over $\Vc $
		such that the vertex operator map $Y(\cdot,z)$ is given by
		\begin{equation}\label{Ymap}
		Y\big(T^+_{[n]} (u)\vac,z\big)=T^+_{[n]} (z|u)\ts T_{[n]}^- (z+hc/2|u)^{-1},
		\end{equation} the vacuum vector is $\vac$
		and the braiding map $\mathcal{S}$ is defined by
		\begin{align}
			&\mathcal{S}(z)\left(R_{nm}^{ 12}(e^{z+u-v })^{-1}\ts T_{[m]}^{+ 24}(v) \ts R_{nm}^{ 12}(e^{z+u-v -h c}) \ts T_{[n]}^{+ 13}(u) (\vac\otimes \vac) \right)\nonumber\\
			=\, & T_{[n]}^{+ 13}(u)\ts R_{nm}^{ 12}(e^{z+u-v +h c})^{-1} \ts
			T_{[m]}^{+ 24}(v) \ts R_{nm}^{ 12}(e^{z+u-v })(\vac\otimes \vac). \label{braiding_r eqs2}
		\end{align}
	\end{theorem}

	Note that~\eqref{braiding_r eqs2} is given over the tensor product
	\[
		\overbrace{\left(\ndo\CC^N\right)^{\otimes n}}^{1} \otimes
		\overbrace{\left(\ndo\CC^N\right)^{\otimes m}}^{2} \otimes
		\overbrace{\Vc}^{3}\otimes
		\overbrace{\Vc }^{4}
	\]
	and that the superscripts $1,2,3,4$ indicate the tensor factors as above. Therefore,
	$T_{[n]}^{+ 13}(u)$ (resp. $T_{[m]}^{+ 24}(v) $) denotes the action of
	$T_{[n]}^{+ }(u)$ (resp. $T_{[m]}^{+ }(v) $) over the tensor factors
	$1,\ldots,n$ and $n+m+1$ (resp. $n+1,\ldots,n+m$ and $n+m+2$).

	\section{Weakly compatible pair associated with \texorpdfstring{$\Vc$}{Vc(glN)}}\label{compatible}

	In this section, we associate a~certain pair of maps with the quantum affine vertex algebra $\Vc$. The next definition generalizes the notion of multiplicative compatible pair from~\cite[Def. 5.1]{BJK}.

	\begin{definition}\label{def_b r_m ult}
		Suppose $\mathcal{V}$ is a~topologically free $\CC[[h]]$-module and
		\[
			\sigma(z),\rho(z)\colon \mathcal{V}\otimes \mathcal{V}\to \mathcal{V}\otimes \mathcal{V}\otimes \CC(z)[[h]]
		\]
		are $\CC[[h]]$-module maps. The pair $(\sigma, \rho)$ is said to be a~\emph{multiplicative weakly compatible pair} if it satisfies the following conditions.
		\begin{enumerate}
			\item The map $\sigma$ satisfies the quantum Yang--Baxter equation,
			\begin{align}
				&\sigma_{12}(z_1)\ts\sigma_{13}(z_1 z_2)\ts\sigma_{23}(z_2)=
				\sigma_{23}(z_2)\ts\sigma_{13}(z_1 z_2)\ts\sigma_{12}(z_1) \label{ybe_m uh}\\
				\intertext{and the unitarity condition,}
				&\sigma(1/z)\ts \sigma_{21}(z)=\sigma_{21}(z)\ts\sigma(1/z) =1.\label{uni_m uh}
			\end{align}
			\item The maps $\sigma(z)$, $\rho(z)$ and $\rho_{21}^{-1}(1/z)$ are regular at $z=0$.

			\item The map
			\begin{align}
				\mathcal{S}_{\sigma,\rho} \colon \mathcal{V}\otimes \mathcal{V} &\to \mathcal{V}\otimes \mathcal{V},\nonumber\\
				\mathcal{S}_{\sigma,\rho} &= \left(\rho (z)\ts \sigma (z)\ts \rho_{21}^{-1}(1/z)\right)\big|_{z=0}\big.\label{ssigmarho}
			\end{align}
			satisfies the quantum Yang--Baxter equation,
			\begin{equation}\label{braiding_r eqs3}
			\left(\mathcal{S}_{\sigma,\rho}\right)_{12}
			\left(\mathcal{S}_{\sigma,\rho}\right)_{13}
			\left(\mathcal{S}_{\sigma,\rho}\right)_{23}
			=
			\left(\mathcal{S}_{\sigma,\rho}\right)_{23}
			\left(\mathcal{S}_{\sigma,\rho}\right)_{13}
			\left(\mathcal{S}_{\sigma,\rho}\right)_{12}.
			\end{equation}
		\end{enumerate}
	\end{definition}

	\begin{remark}
		In this remark, we briefly discuss Definition~\ref{def_b r_m ult}. Regarding~\eqref{ssigmarho}, observe that this identity resembles the form of a~(spectral parameter) Drinfeld-type twist~\cite{Dri1,Dri2,Dri3}. More specifically, $\sigma$ and $\mathcal{S}_{\sigma,\rho}$ play the role of the $R$-matrix and the twisted $R$-matrix and $\rho_{21}(1/z)$ and $\rho(z)$ of the Drinfeld twist elements. The same applies to the rational counterpart of this property, given by~\eqref{Ssigmarhorat} below. Furthermore, it is worth noting that the requirement of regularity at $z=0$ for $\rho_{21}^{-1}(1/z)$, as given by Definition~\ref{def_b r_m ult}, can be regarded as regularity of $\rho_{21}^{-1} $ at infinity. Finally, we remark that our main example of the multiplicative weakly compatible pair satisfies $\rho_{21}^{-1}(1/z)|_{z=0}=1$, so that~\eqref{ssigmarho} collapses to $\mathcal{S}_{\sigma,\rho} = (\rho (z)\ts \sigma (z)) |_{z=0}$ and the twist is one-sided; cf. Proposition~\ref{prop_c p} and its proof.
	\end{remark}

	We shall omit the term ``multiplicative'' and refer to any pair of maps satisfying Definition~\ref{def_b r_m ult} more briefly as a~weakly compatible pair. Note that, by the definition, we implicitly require that the map $\rho$ is invertible, i.e., that there exists a~map
	\[
	\rho^{-1} (z)\colon \mathcal{V}\otimes \mathcal{V}\to \mathcal{V}\otimes \mathcal{V}\otimes\CC (z) [[h]]
	\]
	such that
	$
	\rho (z)\ts\rho^{-1} (z)=\rho^{-1} (z)\ts \rho (z)=1.
	$

	Let us proceed towards the construction of a~weakly compatible pair $(\sigma,\rho)$ over $\Vc$.

	\begin{lemma}
		For any $c \in\CC $ there exists a~unique map
		\[
			\sigma (z)
			\colon \Vc\otimes\Vc \to\Vc\otimes\Vc\otimes \CC(z)[[h]]
		\]
		such that
		\begin{align}
			&\sigma (z)\left(T_{[n]}^{+ 13}(u)\ts T_{[m]}^{+ 24}(v) (\vac\otimes \vac) \right)\nonumber\\
			=\, &\left(\left(\underline{R}_{nm}^{12}\right)^\prime \cdotlr\left(R_{nm}^{ 12}(ze^{u-v})\ts T_{[n]}^{+ 13}(u) R_{[n]}^{1} \ts \underline{R}_{nm}^{21} \right.\right.\nonumber\\
			&\qquad\qquad\quad\left.\left. T_{[m]}^{+ 24}(v) R_{[m]}^{2} \ts \underline{R}_{nm}^{ 21}(ze^{u-v})^{-1}(\vac\otimes \vac)\right)\right) \left(R_{[n]}^{1}R_{[m]}^{2}\right)^{-1}, \label{sigma_m ap}
		\end{align}
		where ``$\cdotlr$'' denotes the standard multiplication in
		$\left(\ndo\CC^N\right)^{\otimes n}\otimes\left(\left(\ndo\CC^N\right)^{\text{op}}\right)^{\otimes m}.$
		Moreover, the map is regular at $z=0$ and satisfies the quantum Yang--Baxter equation~\eqref{ybe_m uh} and the unitarity condition~\eqref{uni_m uh}.
	\end{lemma}

	\begin{proof}
		The fact that~\eqref{sigma_m ap} defines a~$\CC[[h]]$-module map which satisfies the Yang--Baxter equation~\eqref{ybe_m uh} and the unitarity condition~\eqref{uni_m uh} can be verified by suitably adjusting the corresponding parts of the proofs of~\cite[Lemma 4.6]{BK0} and~\cite[Prop. 3.5]{BK}. This is due to the similarity of the forms of $\sigma$ and the maps therein. The proof arguments rely on the Yang--Baxter equations satisfied by the $R$-matrices~\eqref{R},~\eqref{Reu} and~\eqref{R_m ult}. Alternatively, one can prove the aforementioned statements by arguing as in the corresponding part of the proof of~\cite[Thm. 2.3.8]{Gar}, due to similarities between the map $\sigma$ and the Etingof--Kazhdan braiding~\eqref{braiding_r eqs2} which is considered therein. Finally, the regularity of $\sigma$ at $z=0$ is evident from the form of $R(z)$; see~\eqref{rbar_1},~\eqref{rbar_2} and~\eqref{R_m ult}. Furthermore,~\eqref{sigma_m ap} at $z=0$ yields the map
		\begin{align}
			\Vc\otimes\Vc &\to\Vc\otimes\Vc\nonumber\\
			T_{[n]}^{+ 13}(u)\ts T_{[m]}^{+ 24}(v) (\vac\otimes \vac) &\mapsto \left(R_{nm}^{12}\right)^\prime \cdotlr\left(R_{nm}^{ 12} \ts T_{[n]}^{+ 13}(u)\ts R_{nm}^{21} \ts
			T_{[m]}^{+ 24}(v) \ts (R_{nm}^{ 21})^{-1}(\vac\otimes \vac)\right).\label{sigma_z nula_r ev}
		\end{align}
	\end{proof}

	\begin{lemma}
		For any $c \in\CC $ there exists a~unique map
		\[
			\rho (z)
			\colon \Vc\otimes\Vc \to\Vc\otimes\Vc\otimes \CC(z)[[h]]
		\]
		such that
		\begin{align}
			&\rho (z)\left(T_{[n]}^{+ 13}(u)\ts T_{[m]}^{+ 24}(v) (\vac\otimes \vac) \right)\nonumber\\
			=\, &\left(R_{nm}^{21} \right)^\prime \cdotrl\left(T_{[n]}^{+ 13}(u)\ts R_{nm}^{ 12}(ze^{u-v+hc})^{-1} \ts
			T_{[m]}^{+ 24}(v) \ts R_{nm}^{ 12}(z e^{u-v})\ts R_{nm}^{21} (\vac\otimes \vac)\right),\label{rho_m ap}
		\end{align}
		where ``$\cdotrl$'' denotes the standard multiplication in
		$\left(\left(\ndo\CC^N\right)^{\text{op}}\right)^{\otimes n}\otimes\left(\ndo\CC^N\right)^{\otimes m}.$
		Moreover, the map is regular at $z=0$ and invertible. Its inverse is given by
		\begin{align}
			&\rho^{-1} (z)\left(T_{[n]}^{+ 13}(u)\ts T_{[m]}^{+ 24}(v) (\vac\otimes \vac) \right)\nonumber\\
			=\, & \left(R_{nm}^{ 12}(ze^{u-v+hc})^{-1}\right)^\prime \cdotrl\left(T_{[n]}^{+ 13}(u)\ts R_{nm}^{21} \ts
			T_{[m]}^{+ 24}(v) \ts \left(R_{nm}^{21}\right)^{-1} \ts R_{nm}^{ 12}(z e^{u-v})^{-1} (\vac\otimes \vac)\right).\label{rho_i nv_m ap}
		\end{align}
	\end{lemma}

	\begin{proof}
		As with the previous lemma, this one also follows by the usual arguments, so we omit the details. We only remark that~\eqref{rho_m ap} at $z=0$ yields the map
		\begin{align}
			\Vc\otimes\Vc &\to\Vc\otimes\Vc\nonumber \\
			T_{[n]}^{+ 13}(u)\ts T_{[m]}^{+ 24}(v) (\vac\otimes \vac)
			&\mapsto
			\left(R_{nm}^{21} \right)^\prime \cdotrl\left(T_{[n]}^{+ 13}(u) \left(R_{nm}^{ 12}\right)^{-1}
			T_{[m]}^{+ 24}(v) \ts R_{nm}^{ 12} \ts R_{nm}^{21} (\vac\otimes \vac)\right).
			\label{rho_z nula_r ev}
		\end{align}
		Note that the formula~\eqref{rho_i nv_m ap} for the inverse of the map $\rho(z)$ is derived in parallel with the proof of~\cite[Prop. 3.9]{BK}, and the explicit expression for $ (R_{nm}^{ 12}(ze^{u-v+hc})^{-1})^\prime$ can be obtained by using the crossing symmetry properties~\eqref{csym_m}.
	\end{proof}
	Throughout the paper, we are interested in weakly compatible pairs on $h$-adic quantum vertex algebras which are in tune with their braiding map in the following sense.

	\begin{definition}\label{weaklyassoc}
		Let $(\mathcal{V},Y,\Sc,\vac)$ be an $h$-adic quantum vertex algebra and $(\sigma,\rho)$ a~weakly compatible pair over $\mathcal{V}$. The pair $(\sigma,\rho)$ is said to be \emph{weakly associated} with the $h$-adic quantum vertex algebra $\mathcal{V}$
		if there exists a~map
		\[
			\widehat{\Sc}(z)\colon \mathcal{V}\otimes\mathcal{V}\to\mathcal{V}\otimes\mathcal{V}\otimes\CC(z)[[h]],
		\]
		which is regular at $z=0$, such that we have
		\begin{equation}\label{braiding_r eqs}
		\Sc(z)=\iota_z \, \widehat{\Sc}(x)\big|_{x=e^z} \big.
		\Fand
		\Sc_{\sigma,\rho} = \widehat{\Sc}(z)\big|_{z=0}\big.,
		\end{equation} where the map $\mathcal{S}_{\sigma,\rho}$ is defined by~\eqref{ssigmarho} and $\iota_z$ is the embedding~\eqref{iotau_r ev}.
	\end{definition}

	Definition~\ref{weaklyassoc} generalizes the notion of compatible pair associated with $h$-adic quantum vertex algebra~\cite[Def. 5.3]{BJK}. In fact, by a~simple comparison, one observes that
	Definitions~\ref{def_b r_m ult} and~\ref{weaklyassoc} are obtained from~\cite[Defs. 5.1, 5.3]{BJK} by setting $z=0$ in the requirements which compare the pair $(\sigma,\rho)$ and the braiding $\Sc$ of the corresponding quantum vertex algebra and by omitting the unitarity requirement on $\mathcal{S}_{\sigma,\rho}$, which is not used in the sequel. As we shall see, such a~weakening of the original structure provides a~framework for establishing a~connection between $\Vc$ and the representations of the quantized enveloping algebra
	$U_h (\mathfrak{gl}_N)$ and the reflection equation algebra $\mathcal{O}_h (Mat_N)$, which is motivated by the evaluation homomorphism
	$U_h (\widehat{\mathfrak{gl}}_N)\to U_h (\mathfrak{gl}_N)$.
	\begin{proposition}\label{prop_c p}
		The maps $\sigma$ and $\rho$, defined by~\eqref{sigma_m ap} and~\eqref{rho_m ap} respectively, form a~weakly compatible pair. This pair is weakly associated with $\Vc$.
	\end{proposition}

	\begin{proof}
		Let
		$
		\widehat{\Sc}(z)\colon \Vc\otimes\Vc\to\Vc\otimes\Vc\otimes\CC(z)[[h]]
		$
		be the $\CC[[h]]$-module map from~\cite[Lemma 4.3]{BJK},
		\begin{align*}
			&\widehat{\Sc}(z)\big(R_{nm}^{ 12}(ze^{u-v})^{-1} T_{[m]}^{+ 24}(v)
			R_{nm}^{ 12}(ze^{u-v-h c}) T_{[n]}^{+ 13}(u)(\vac\otimes \vac) \big)\\
			=\, & T_{[n]}^{+ 13}(u) R_{nm}^{ 12}(ze^{u-v+h c})^{-1}
			T_{[m]}^{+ 24}(v) R_{nm}^{ 12}(ze^{u-v})(\vac\otimes \vac).
		\end{align*}
		Clearly, at $z=0$ the map $\widehat{\Sc}(z)$ is regular and gives rise to the map
		\begin{align}
			\Vc\otimes\Vc &\to\Vc\otimes \Vc\nonumber\\
			T_{[n]}^{+ 13}(u)\ts T_{[m]}^{+ 24}(v)
			(\vac\otimes \vac) & \mapsto
			\left(R_{nm}^{ 12}\right)^\prime \cdotlr \left(R_{nm}^{ 12} \ts T_{[n]}^{+ 13}(u) \left(R_{nm}^{ 12}\right)^{-1}
			T_{[m]}^{+ 24}(v) R_{nm}^{ 12} (\vac\otimes \vac)\right).\label{shatrev}
		\end{align}
		Moreover, it satisfies the first identity in~\eqref{braiding_r eqs}; cf.~\eqref{braiding_r eqs2}. As for the second identity in~\eqref{braiding_r eqs}, one derives from~\eqref{rho_i nv_m ap} that $\rho^{-1}_{21}(1/z)$ at $z=0$ produces the identity on $\Vc\otimes\Vc$. Therefore, it suffices to show that the equality $\widehat{\mathcal{S}}(z)\big|_{z=0} = \left(\rho (z)\ts \sigma (z) \right)\big|_{z=0}$ holds. This follows from the explicit formulas~\eqref{sigma_z nula_r ev},~\eqref{rho_z nula_r ev} and~\eqref{shatrev}. Finally, as
		$\widehat{\Sc}(z) $ satisfies the quantum Yang--Baxter equation, by setting $z=0$, we find that the map $\Sc_{\sigma,\rho} = \widehat{\Sc}(z)\big|_{z=0}\big.$ satisfies~\eqref{braiding_r eqs3}, as required.
	\end{proof}

	\begin{remark}\label{Yrho_r em}
		Consider the map
		\[
			\rho (x)\big|_{x=e^z}\colon \Vc\otimes\Vc \to\Vc\otimes \Vc\otimes \CC_* (z)[[h]].
		\]
		By applying the embedding $\iota_z\colon \CC_* (z)\to \CC((z))$ to $ \rho (x)\big|_{x=e^z}$, we get
		\[
			\rho (e^z):= \iota_z \left(\rho (x)\big|_{x=e^z}\right)
			\colon \Vc\otimes\Vc \to\Vc\otimes\Vc\otimes \CC((z))[[h]].
		\]
		The composition of the vertex operator map~\eqref{Ymap} and $\rho (e^z)$,
		\begin{equation}\label{yrho_r ev}
		Y^\rho(z):= Y(z)\ts \rho(e^z)
		\colon \Vc\otimes\Vc \to\Vc\otimes \CC [[z,h]]
		\end{equation} maps $T_{[n]}^{+ 13}(u)\ts T_{[m]}^{+ 24}(v) (\vac\otimes \vac)$ to
		\begin{equation}\label{Ymap rho}
		\left(R_{nm}^{ 21} \right)^\prime \cdotrl\left(T_{[n]}^{+ 13}(z |u)\ts
		T_{[m]}^{+ 24}(v)\ts R_{nm}^{ 21} (\vac\otimes \vac)\right).
		\end{equation}
		Indeed, this is an immediate consequence of the RTT-relation~\eqref{RTT3} and the explicit formula~\eqref{Ymap} for the vertex operator map $Y(\cdot, z)$. This demonstrates the key property of the map $\rho$ to cancel all singularities of the vertex operator map $Y (z)$, as the image of $Y^\rho (z)$ possesses only nonnegative powers of the variable $z$; cf.~\eqref{yrho_r ev}.
	\end{remark}

	In general, let $(\mathcal{V},Y,\vac,\mathcal{S})$ be an $h$-adic quantum vertex algebra and $(\sigma,\rho)$ a~weakly compatible pair over $\mathcal{V}$. Following Remark~\ref{Yrho_r em}, one can associate with $Y(\cdot,z)$ the map
	\begin{equation}\label{Yrhho}
	Y^\rho:= Y(z)\ts \rho(e^z)
	\colon \mathcal{V}\otimes\mathcal{V} \to\mathcal{V}\otimes \CC((z))[[h]].
	\end{equation}

	\section[Deformed $\phi$-coordinated $\Vc$-modules]{Deformed \texorpdfstring{$\phi$}{phi}-coordinated \texorpdfstring{$\Vc$}{Vc(glN)}-modules associated to \texorpdfstring{$U_h (\gl_N)$}{Uh(glN)} and \texorpdfstring{$\mathcal{O}_h (Mat_N)$}{Oh(MatN)}}\label{section_0 5}

	In this section we establish a~representation-theoretic connection between the quantum vertex algebra $\Vc$ and the algebras $U_h (\gl_N)$ and $\mathcal{O}_h (Mat_N)$ by using the structure of weak
	$(\sigma,\rho)$-deformed $\phi$-coordinated module. As before, $(\sigma,\rho)$ is the pair from Proposition~\ref{prop_c p}. First, we generalize~\cite[Def. 6.1]{BJK}, which is based on~\cite[Def. 5.1]{BK}, to introduce the aforementioned notion. As with the original definition, we only consider the associate $\phi(z_2,z_0)=z_2 e^{z_0}\in\CC[[z_0,z_2 ]]$; cf.~\cite{Liphi}.

	\begin{definition}
		Let $(\mathcal{V},Y,\vac,\Sc)$ be an $h$-adic quantum vertex algebra and $(\sigma,\rho)$ a weakly compatible pair weakly associated with $\mathcal{V}$. Let $W$ be a~topologically free $\CC[[h]]$-module and
		\begin{align*}
			Y_W (\cdot,z) \colon \mathcal{V}\otimes W&\to W((z))_h, \\
			v\otimes w&\mapsto Y_W (z)(v\otimes w)=Y_W (v,z)w=\sum_{r\in\ZZ} v_{r } w \ts z^{-r-1 } \nonumber
		\end{align*}
		a $\CC[[h]]$-module map. A pair $(W,Y_W)$ is said to be a~\emph{weak $(\sigma,\rho)$-deformed $\phi$-coordinated $\mathcal{V}$-module} if the map $Y_W (\cdot, z)$ possesses the following properties:
		\begin{align}
			&Y_W (\vac,z)w=w\quad\text{for all }w\in W;\nonumber
			\intertext{the \emph{weak $\rho $-associativity}:
			for any $u,v \in \mathcal{V}$ and $n \in\mathbb{Z}_{> 0}$
			there exists $r\in\mathbb{Z}_{\geqslant 0}$
			such that}
			&(z_1 -z_2)^r\ts Y_W (u,z_1)\ts Y_W (v,z_2)\in \om\left(W,W((z_1,z_2))\right) \mod h^n,\label{rho-assoc1}\\
			&\left((z_1 -z_2)^r\ts\ts Y_W (u, z_1) Y_W (v,z_2) \right)\Big|_{z_1 =z_2 e^{z_0 }}^{\text{mod } h^n} \Big.\label{rho-assoc2}\\
			&\qquad\big. - z_2^r (e^{z_0} -1)^r\ts Y_W\big(Y^{\rho } (u, z_0)v,z_2\big)\,\in\, h^n \om\left(W,W[[z_0^{\pm 1},z_2^{\pm 1}]]\right);\nonumber
			\intertext{the \emph{$\sigma $-locality}:
			for any $u,v\in \mathcal{V}$ and $n\in\mathbb{Z}_{> 0}$ there exists
			$r\in\mathbb{Z}_{\geqslant 0}$ such that }
			&\big((z_1 -z_2)^r\ts Y_W (z_1)\big(1\otimes Y_W (z_2)\big)\ts\iota_{z_1,z_2}\big(\sigma (z_1 /z_2)(u\otimes v)\otimes w\big)\big.
			\nonumber\\
			&\qquad\big. - (z_1 -z_2)^r\ts Y_W (v,z_2) Y_W (u, z_1) w \big) \,
			\in\, h^n\ts W[[z_1^{\pm 1},z_2^{\pm 1}]]\quad \text{for all } w\in W, \label{sigma-loc}
		\end{align}
		where $\iota_{z_1,z_2}$ is the embedding
		$\CC(z_1,z_2)[[h]]\hookrightarrow \CC((z_1))((z_2))[[h]]$.
	\end{definition}

	Regarding the weak $\rho$-associativity property in the definition above, note that~\eqref{rho-assoc2} indicates that the substitution $z_1 =z_2 e^{z_0}$ is applied to the corresponding expression when regarded modulo $h^n$, while~\eqref{rho-assoc1} ensures that such a~substitution exists.

	Let
	\[
		L^+ (z|u)= L^- + L^+ hz^{-1}e^{-u } \in\ndo\CC^N\otimes U_h (\gl_N)[[ u ]][z^{-1}].
	\]
	Due to~\cite[Eq. 5.5]{JLM}, the series
	\begin{equation}\label{Lzu}
	L (z|u)= L^+ (z|u)(L^-)^{-1} = 1 + L^+ (L^-)^{-1} hz^{-1}e^{-u }
	\end{equation} satisfies the \emph{reflection equation},
	\begin{equation}\label{reflecteq}
	R(z_1 e^{u-v}/z_2)\ts L_1 (z_1 |u)\ts R_{21}\ts L_2 (z_2 |v)
	=L_2 (z_2 |v)\ts R\ts L_1 (z_1 |u)\ts R_{21}(z_1 e^{u-v}/z_2).
	\end{equation}

	\begin{remark}
		It is worth noting that the reflection equation~\eqref{reflecteq} possesses a~somewhat nonstandard form, as both of its sides contain one constant and one parameter-dependent $R$-matrix. Such reflection equations go back to the works of Cherednik~\cite{Che}, Freidel and Maillet\cite{FreMai} and Kulish and Sklyanin~\cite{KulSkl}. Note that this reflection equation governs the form of the map $\sigma$, as~\eqref{sigma_m ap} for $m=n=1$ can be written equivalently as
		\[
			R (ze^{u-v})\ts T_{13}^{+ }(u)\ts R_{21} \ts
			T_{24}^{+ }(v) (\vac\otimes \vac) =
			\sigma (z)\left(T_{ 24}^{+}(v)\ts R \ts T_{13}^{+ }(u) \ts R_{ 21}(ze^{u-v}) (\vac\otimes \vac) \right).
		\]
	\end{remark}

	Next, for any $n=1,2,\ldots$ and the family of variables $u=(u_1,\ldots,u_n)$, define
	\[
		L^+_{[n]}(z|u)=L^+_1 (z|u_1)\cdots L^+_n (z|u_n)
		\Fand
		L_{[n]}^- =L_1^-\cdots L_n^-.
	\]
	Finally, introduce the operator
	\[
		L_{[n]}(z|u)=L^+_{[n]}(z|u)\big(L_{[n]}^-\big)^{-1},
	\]
	which belongs to
	$\left(\ndo\CC^N\right)^{\otimes n}\otimes \ndo U_h (\gl_N)[[u_1,\ldots,u_n ]][z^{-1}]$, where its action is given by the algebra multiplication. It can be regarded as a~normal ordering of $n$ copies of the operators~\eqref{Lzu}. In the next lemma, we generalize the reflection equation~\eqref{reflecteq}.

	\begin{lemma}
		For any $n$ and $u=(u_1,\ldots,u_n)$ we have
		\begin{equation}
		L_{[n]}(z|u)=\prod_{i=1,\dots,n }^{\longrightarrow}
		\left(L_i (z|u_i)R_{i+1\ts i}\ldots R_{n i}\right)
		R_{[n]}^{-1},
		\end{equation} where $R_{[n]}$ is given by~\eqref{rev3_R n}. Moreover, for any $m$ and $v=(v_1,\ldots,v_m)$, we have the equality
		\begin{align}\label{qceq}
			&R_{nm}^{12}(z_1 e^{u-v}/z_2)\ts
			L_{[n]}^{13}(z_1 |u)\ts
			R_{[n]}^1\ts
			\underline{R}_{nm}^{21}\ts
			L_{[m]}^{23}(z_2 |v)\ts
			R_{[m]}^2 \nonumber\\
			=&\,L_{[m]}^{23}(z_2 |v)\ts R_{[m]}^2\ts
			\underline{R}_{nm}^{12}\ts
			L_{[n]}^{13}(z_1 |u)\ts R_{[n]}^1\ts \underline{R}_{nm}^{21}(z_1 e^{u-v}/z_2)
		\end{align}
		of formal power series with coefficients in $\left(\ndo\CC^N\right)^{\otimes n}\otimes \left(\ndo\CC^N\right)^{\otimes m}\otimes \ndo U_h (\gl_N) $.
	\end{lemma}

	\begin{proof}
		Both assertions of the lemma can be verified by direct computations which rely on the reflection equation~\eqref{reflecteq} and the Yang--Baxter equations satisfied by the $R$-matrices~\eqref{R} and~\eqref{R_m ult}. In addition, the proof of the first assertion makes use of the equality
		\begin{equation}\label{start01}
		\left(L_1^-\right)^{-1} R_{21}\ts L_2^\pm
		=
		L_2^\pm\ts R_{21} \left(L_1^-\right)^{-1},
		\end{equation} which can be proved by combining the identities in~\eqref{rlls}.
	\end{proof}
	We are now ready to equip the quantized enveloping algebra with the structure of weak $(\sigma,\rho)$-deformed $\phi$-coordinated $\Vc$-module.

	\begin{theorem}\label{module_t rig}
		Let $c\in\CC $. There exists a~unique structure of weak $(\sigma,\rho)$-deformed $\phi$-coordinated $\Vc$-module over $ U_h (\gl_N)$ such that the module map $Y_{U_h (\gl_N)}(\cdot,z)$ satisfies
		\begin{equation}\label{def_m od_m ap}
		Y_{U_h (\gl_N)}(T_{[n]}^+ (u)\vac,z) = L^+_{[n]}(z|u) \big(L_{[n]}^-\big)^{-1}.
		\end{equation}
	\end{theorem}

	\begin{proof}
		First, we prove that the map $Y_{U_h (\gl_N)}(\cdot,z)$, as defined by~\eqref{def_m od_m ap}, possesses the weak $\rho$-associativity properties~\eqref{rho-assoc1} and~\eqref{rho-assoc2}. Consider the expression
		\begin{equation}\label{start}
		T_{[n]}^{+13}(u)\ts T_{[m]}^{+24}(v) (\vac\otimes\vac),
		\end{equation} where $u=(u_1,\ldots,u_n)$ and $v=(v_1,\ldots,v_m)$. Note that its coefficients belong to the tensor product
		\[
			\overbrace{\left(\ndo\CC^N\right)^{\otimes n}}^{1} \otimes
			\overbrace{\left(\ndo\CC^N\right)^{\otimes m}}^{2} \otimes
			\overbrace{\Vc}^{3}\otimes
			\overbrace{\Vc }^{4}
		\]
		and that its superscripts indicate the tensor factors as above. Its image under the map $Y_{U_h (\gl_N)}(\cdot, z_1)(1\otimes Y_{U_h (\gl_N)}(\cdot, z_2))$ equals
		\begin{equation}\label{start5}
		L^{+13}_{[n]}(z_1 |u) \big(L_{[n]}^{-13}\big)^{-1}
		L^{+23}_{[m]}(z_2 |v) \big(L_{[m]}^{-23}\big)^{-1},
		\end{equation} which belongs to
		\[
			\left(\ndo\CC^N\right)^{\otimes (n+m)} \otimes
			\om\left(U_h (\gl_N), U_h (\gl_N)[[u_1,\ldots,u_n,v_1,\ldots,v_m ]][z_1^{-1},z_2^{-1}]\right).
		\]
		Clearly, this implies~\eqref{rho-assoc1}.

		Let us verify~\eqref{rho-assoc2}. The image of~\eqref{start} under $Y^\rho(z_0)$ equals
		\begin{equation}\label{start3}
		\left(R_{nm}^{ 21} \right)^\prime \cdotrl\left(T_{[n]}^{+ 13}(z_0 |u)\ts
		T_{[m]}^{+ 24}(v)\ts R_{nm}^{ 21} (\vac\otimes \vac)\right),
		\end{equation} as evident from~\eqref{Ymap rho}. By applying $Y_{U_h (\gl_N)}(\cdot,z_2)$ to~\eqref{start3}, we get
		\[
			\left(R_{nm}^{ 21} \right)^\prime \cdotrl\left(L_{[n]}^{+ 13}(z_2 |z_0 +u)\ts
			L_{[m]}^{+ 23}(z_2 |v)
			\big(L_{[m]}^{- 23}\big)^{-1}
			\big(L_{[n]}^{- 13}\big)^{-1}
			R_{nm}^{ 21} (\vac\otimes \vac)\right).
		\]
		Due to the first identity in~\eqref{rlls} for $L^-$, this equals
		\begin{equation}\label{start4}
		\left(R_{nm}^{21} \right)^\prime \cdotrl\left(L_{[n]}^{+ 13}(z_2 |z_0 +u)\ts
		L_{[m]}^{+ 23}(z_2 |v)
		R_{nm}^{ 21}
		\big(L_{[n]}^{- 13}\big)^{-1} \big(L_{[m]}^{- 23}\big)^{-1}
		(\vac\otimes \vac)\right).
		\end{equation}
		Next, we use~\eqref{start01}
		to rewrite~\eqref{start4} as
		\[
			\left(R_{nm}^{21} \right)^\prime \cdotrl\left(L_{[n]}^{+ 13}(z_2 |z_0 +u)
			\big(L_{[n]}^{- 13}\big)^{-1}
			R_{nm}^{ 21} \ts
			L_{[m]}^{+ 23}(z_2 |v)
			\big(L_{[m]}^{- 23}\big)^{-1}
			(\vac\otimes \vac)\right).
		\]
		Cancelling the terms $\left(R_{nm}^{21} \right)^\prime$ and $ R_{nm}^{21}$ we get
		\begin{equation}\label{start6}
		L_{[n]}^{+ 13}(z_2 |z_0 +u)
		\big(L_{[n]}^{- 13}\big)^{-1}
		L_{[m]}^{+ 23}(z_2 |v)
		\big(L_{[m]}^{- 23}\big)^{-1}
		(\vac\otimes \vac).
		\end{equation}
		It remains to observe that the substitution $z_1 =z_2 e^{z_0}$ turns the expression in~\eqref{start5} to~\eqref{start6}, which implies~\eqref{rho-assoc2}.

		Let us prove the $\sigma$-locality property~\eqref{sigma-loc}. We shall use the expression for the vertex operator map,
		\[
			Y_{U_h (\gl_N)}(T_{[n]}^+ (u)\vac,z) = L_{[n]}(z|u),
		\]
		which is equivalent to~\eqref{def_m od_m ap}. It suffices to check that the images of~\eqref{start} under
		\[
			Y_{U_h (\gl_N)}(z_1)\left(1\otimes Y_{U_h (\gl_N)}(z_2)\right)\iota_{z_1,z_2}\ts\sigma(z_1 /z_2)
			\Fand
			Y_{U_h (\gl_N)}(z_2)\left(1\otimes Y_{U_h (\gl_N)}(z_1)\right)P,
		\]
		where $P$ is the permutation operator~\eqref{permutation}, coincide. Set $z=z_1 /z_2$. By~\eqref{sigma_m ap} and~\eqref{def_m od_m ap}, the image of~\eqref{start} under the first map is
		\begin{align}
			&\iota_{z_1,z_2} \left(
			\left(\underline{R}_{nm}^{12}\right)^\prime \cdotlr
			\big(R_{nm}^{ 12}(ze^{u-v})\ts
			L_{[n]}^{ 13}(z_1 |u)R_{[n]}^1\ts
			\underline{R}_{nm}^{21} \big.\right.\nonumber \\
			&\Big.\big.\qquad\qquad\qquad L_{[m]}^{ 23}(z_2 |v)R_{[m]}^2\ts
			\underline{R}_{nm}^{ 21}(ze^{u-v})^{-1}(\vac\otimes \vac)\big)\Big)
			\left(R_{[n]}^1 R_{[m]}^2\right)^{-1},
			\label{looc_l hs}
		\end{align}
		while the image under the second map is
		\begin{equation}\label{looc_r hs}
		L_{[m]}^{ 23}(z_2 |v)\ts
		L_{[n]}^{ 13}(z_1 |u)
		(\vac\otimes \vac).
		\end{equation}
		Hence, the identity~\eqref{qceq} implies that~\eqref{looc_l hs} is equal to~\eqref{looc_r hs}, as required.
	\end{proof}
	Suppose $W$ is a~$ U_h (\gl_N)$-module. We denote by $L^+_{[n]}(z|u)_W$ (resp. $\big(L_{[n]}^-\big)^{-1}_W$) the action of $L^+_{[n]}(z|u) $ (resp. $\big(L_{[n]}^-\big)^{-1} $) on $W$. Hence, we have
	\begin{align*}
		&L^+_{[n]}(z|u)_W\in(\ndo\CC^N)^{\otimes n}\otimes\om(W,W[[u_1,\ldots,u_n ]][z^{-1}]),\\
		&
		\big(L_{[n]}^-\big)^{-1}_W\in(\ndo\CC^N)^{\otimes n}\otimes\ndo W.
	\end{align*}
	By arguing as in the proof of Theorem~\ref{module_t rig}, one can verify the next corollary.
	\begin{corollary}\label{kor_4 4}
		Let $c\in\CC $. Suppose $W$ is a~topologically free $\CC[[h]]$-module and also a~$ U_h (\gl_N)$-module. There exists a~unique structure of weak $(\sigma,\rho)$-deformed $\phi$-coordinated $\Vc$-module over $W$ such that the corresponding module map $Y_{W}(\cdot,z)$ satisfies
		\[
			Y_{W}(T_{[n]}^+ (u)\vac,z) = L^+_{[n]}(z|u)_W \big(L_{[n]}^-\big)^{-1}_W.
		\]
	\end{corollary}
	It is worth noting that, in the proof of Theorem~\ref{module_t rig}, we showed that the properties~\eqref{rho-assoc2} and~\eqref{sigma-loc} hold for $r=0$. This is a~consequence of the form of the map~\eqref{def_m od_m ap},
	\[
		Y_{U_h (\gl_N)}(v,z)u\in U_h (\gl_N)[z^{-1}]_h\quad\text{for all }v\in\Vc,\, u\in U_h (\gl_N).
	\]
	Naturally, the same applies to Corollary~\ref{kor_4 4}. Suppose $(W,Y_W)$ is a~weak $(\sigma,\rho)$-deformed $\phi$-coordinated $\Vc$-module. If the module map $Y_W (\cdot, z)$ satisfies
	\begin{equation}\label{sigma-comm}
	Y_W (z_1)\big(1\otimes Y_W (z_2)\big)\big(\sigma (z_1 /z_2)(u\otimes v)\otimes w\big)= Y_W (v,z_2) Y_W (u, z_1) w
	\end{equation} for all $u,v\in \Vc$ and $w\in W$, it is said to be \emph{$\sigma$-commutative}.

	Next, we turn our attention to the reflection equation algebra.

	\begin{theorem}\label{main_t thm}
		Let $c\in\CC $. Suppose that $(W,Y_W)$ is a~weak $(\sigma,\rho)$-deformed $\phi$-coordinated $\Vc$-module such that for all $i,j=1,\ldots,N$ we have
		\begin{equation}\label{condition}
		Y_W (t_{ij}^{(-1)}\vac,z)\in z^{-1}\ndo W [[z]].
		\end{equation}
		Then the assignments
		\begin{equation}\label{sigma-commm}
		\left(\ell_{ij}\right)_W = \rez_z Y_W (t_{ij}^{(-1)}\vac,z)
		\end{equation} with $i,j=1,\ldots,N$
		define a~structure of $\mathcal{O}_h (Mat_N)$-module over $W$.
	\end{theorem}

	\begin{proof}
		Let $(W,Y_W)$ be a~weak $(\sigma,\rho)$-deformed $\phi$-coordinated $\Vc$-module such that condition~\eqref{condition} holds. It suffices to show that the matrix
		\[
			\mathcal{L} =\sum_{i,j=1,\ldots,N} e_{ij}\otimes \left(\ell_{ij}\right)_W \in\ndo\CC^N\otimes\ndo W,
		\]
		with $\left(\ell_{ij}\right)_W$ given by~\eqref{sigma-commm}, satisfies the reflection relation~\eqref{def_r efl}. By~\eqref{condition}, the operator $\mathcal{L}^\prime (z) =-Y_W (T^+ (0)\vac, z)$ is $\sigma$-commutative with itself, i.e., it satisfies the $\sigma$-locality property~\eqref{sigma-loc} for $r=0$; recall~\eqref{sigma-comm}. Next, due to~\eqref{sigma-commm}, we have
		$
		\rez_z \mathcal{L}^\prime (z) =
		h\mathcal{L}$.

		The $\sigma$-commutativity relation for $\mathcal{L}^\prime (z) $ can be written using the explicit expression~\eqref{sigma_m ap} for the braiding map $\sigma$ as
		\[
			R^\prime\cdotlr\left(R(z_1 /z_2)\ts
			\mathcal{L}^\prime_{13} (z_1)
			R_{21}\ts
			\mathcal{L}^\prime_{23} (z_2)
			R_{21}(z_1 /z_2)^{-1}
			\right)
			=\mathcal{L}^\prime_{23} (z_2) \mathcal{L}^\prime_{13} (z_1).
		\]
		Clearly, this is equivalent to
		\begin{equation}\label{dh1}
		R(z_1 /z_2)
		\left(z_1\mathcal{L}^\prime_{13} (z_1) \right)
		R_{21}
		\left(z_2\mathcal{L}^\prime_{23} (z_2) \right)
		=
		\left(z_2\mathcal{L}^\prime_{23} (z_2) \right)
		R
		\left(z_1\mathcal{L}^\prime_{13} (z_1) \right)
		R_{21}(z_1 /z_2).
		\end{equation}
		Setting
		\begin{equation}\label{R-matrix}
		R(z_1 /z_2) = \R(z_1 /z_2) =R+\frac{(e^{h/2}-e^{-h/2})z_1 /z_2}{1-z_1 /z_2} P
		\end{equation} in~\eqref{dh1} and then extracting the constant terms with respect to the variable $z_1$ yields
		\begin{equation}\label{R-matrix-2}
		R \ts
		\mathcal{L}^\prime_{13} \ts
		R_{21}
		\left(z_2\mathcal{L}^\prime_{23} (z_2) \right)
		=
		\left(z_2\mathcal{L}^\prime_{23} (z_2) \right)
		R \ts
		\mathcal{L}^\prime_{13} \ts
		R_{21}
		\end{equation} for $\mathcal{L}^\prime =\rez_z \mathcal{L}^\prime (z) $. Indeed, we can assume that the $R$-matrix is as in~\eqref{R-matrix}, as the additional normalization terms of $R(z_1 /z_2)$ and $R_{21}(z_1 /z_2)$ in~\eqref{dh1} cancel. Next, by extracting the constant terms with respect to the variable $z_2$
		in~\eqref{R-matrix-2}, we obtain the reflection relation
		\[
			R \ts
			\mathcal{L}^\prime_{13}\ts
			R_{21}\ts
			\mathcal{L}^\prime_{23}
			=
			\mathcal{L}^\prime_{23} \ts
			R \ts
			\mathcal{L}^\prime_{13} \ts
			R_{21}.
		\]
		As $\mathcal{L}^\prime=h\mathcal{L}$, with $W$ being torsion-free, we conclude that the operator $\mathcal{L}$ also satisfies the reflection relation~\eqref{def_r efl}, as required.
	\end{proof}
	It is worth noting that the weak $(\sigma,\rho)$-deformed $\phi$-coordinated $\Vc$-modules obtained by the construction from Corollary~\ref{kor_4 4} satisfy the requirement imposed by~\eqref{condition}.

	At the end, we discuss a~relation between the quantum Feigin--Frenkel center and the central elements in the quantized enveloping algebra from the viewpoint of weak deformed $\phi$-coordinated modules. First, we follow the exposition in~\cite[Sect. 2, 3]{JLM0} to recall the fusion procedure for the Hecke algebra~\cite{C,IMO,N}. Let $\Lambda$ be a~standard tableau of shape $\lambda\vdash n$ with at most $N$ rows and $c_k (\Lambda)$ the content $j-i$ of the box $(i,j)$ of $\lambda$ occupied by $k$ in $\Lambda $. Introduce the order over the set of all pairs $(i,j)$, where
	$1\leqslant i <j\leqslant n $, by
	\begin{equation}\label{order}
	(i,j)\prec (i^\prime, j^\prime)
	\qquad\text{if}\qquad j<j^\prime
	\quad\text{or}\quad j=j^\prime\text{ and }i<i^\prime.
	\end{equation}
	Let
	\begin{equation}\label{givenproduct}
	\R_\Lambda(z_1,\ldots,z_n)
	=
	\prod_{(i,j) }^{\longrightarrow}
	\left(
	P_{j-i\ts j-i+1}\ts \R_{j-i\ts j-i+1}(z_i e^{ (c_i (\Lambda)-c_j (\Lambda))h}/z_j)
	\right),
	\end{equation} where the product is taken over the set of all pairs $(i,j)$ such that $1\leqslant i <j\leqslant n$ and the arrow indicates that the factors are ordered with respect to~\eqref{order}. Let $\lambda^\prime$ be the conjugate partition of $\lambda,$ $c_{\lambda^\prime}$ the corresponding Schur element and $\check{R}_0$ a~certain invertible operator; see~\cite[Sect. 3]{JLM0} for details. By the fusion procedure,
	\begin{equation}\label{versionthree}
	\Ec_\Lambda :=\frac{1}{c_{\lambda^\prime}}\ts \R_\Lambda(z_1,\ldots,z_n)\ts
	\check{R}_0^{-1}\Big|_{z_1 =1}\Big|_{z_2 =1}\ldots \Big|_{z_n =1}
	\end{equation} is a~well-defined operator satisfying $\Ec_\Lambda^2 =\Ec_\Lambda.$

	Denote by $\Vn$ the quantum affine vertex algebra $\Vc$ at the critical level $c=-N$ and let $D $ be the diagonal matrix~\eqref{diagonal}. Consider the formal power series
	\begin{equation}\label{telambda}
	T_\Lambda(u)=\tr_{1,\ldots,n} \, T_{[n]}^+ (u^{(\Lambda)})\vac \ts D^{\otimes n}\ts \Ec_\Lambda \in \Vn[[u]],
	\end{equation} where the trace is taken over the tensor factors $1,\ldots,n$. Due to~\cite[Prop. 7.4]{BJK}, all coefficients of $T_\Lambda(u)$ belong to the center of the quantum vertex algebra $\Vn$.

	By using Corollary~\ref{kor_4 4} with $W=U_h (\gl_N)$, one finds that the image of the series~\eqref{telambda} under the map $Y_{U_h (\gl_N)}(\cdot,z)$ is equal to
	\begin{align}
		\tr_{1,\ldots,n} \,
		L^+_{[n]}(z|u^{(\Lambda)}) \big(L_{[n]}^-\big)^{-1}\ts D^{\otimes n}\ts \Ec_\Lambda
		\in U_h (\gl_N)[[u]][z^{-1}]\label{yuhgln},
	\end{align}
	where $u^{(\Lambda)}=(u-c_1 (\Lambda)h,\ldots,u-c_n (\Lambda)h)$. In particular, setting $u=0$ in~\eqref{yuhgln} and then multiplying the resulting expression by $z^n$, one obtains the polynomial from $U_h (\gl_N) [z ]$,
	\begin{align}
		\tr_{1,\ldots,n}
		\left(zL_1^- + he^{c_1 (\Lambda) h}L_1^+\right)\ldots
		\left(zL_n^- + he^{c_n (\Lambda) h} L_n^+ \right)
		\left(L_n^-\right)^{-1}\ldots
		\left(L_1^-\right)^{-1}
		D^{\otimes n}\ts \Ec_\Lambda.\label{yuhgln1}
	\end{align}
	Finally, we remark that all coefficients of the polynomial~\eqref{yuhgln1} belong to the center of the quantized enveloping algebra $U_h (\gl_N) $. Indeed, this follows from~\cite[Thm. 3.1]{JLM}, as the polynomial coincides, up to a~multiplicative factor, with the $q$-analogues $\mathbb{S}_\lambda(z)$ of Okounkov's quantum immanants~\cite{Ok} introduced therein.

	\begin{example}
		This example illustrates two particular instances of~\eqref{yuhgln1} in the case of $U_h (\mathfrak{gl}_2)$. Consider the standard tableaux
		\begin{equation}\label{statableaux}
		\Lambda_1 =\ytableausetup{centertableaux}
		\begin{ytableau}
			1
		\end{ytableau}
		\Fand
		\Lambda_2 =\ytableausetup{centertableaux}
		\begin{ytableau}
			1 \\
			2
		\end{ytableau}
		\end{equation} of shapes
		$\lambda_1 =(1)$ and $\lambda_2 =(1,1)$. Their contents are
		\begin{equation}\label{contents}
		c_1 (\Lambda_1)=c_1 (\Lambda_2)=1-1=0\Fand c_2 (\Lambda_2)=1-2= -1,
		\end{equation} while the products of hook lengths are
		\begin{equation}\label{hooks}
		h(\lambda_1)=1\Fand h(\lambda_2)= 1\cdot 2=2.
		\end{equation}
		The corresponding idempotents, obtained by the fusion procedure~\eqref{versionthree}, are
		\begin{equation}\label{elamba12trig}
		\mathcal{E}_{\Lambda_1} =I\in\ndo \CC^2\fand
		\mathcal{E}_{\Lambda_2} =\textstyle\frac{1}{1+e^h} \left(I-e^{h/2}R^{-1}P\right)\in\ndo \CC^2\otimes \ndo \CC^2 [[h]],
		\end{equation} where $I$ denotes the identity (on $\CC^2$ and also on $\CC^2\otimes\CC^2$), $R$ the operator defined by~\eqref{R} and $P$ the permutation operator~\eqref{permutation}, both for $N=2$.

		The above expression for $\mathcal{E}_{\Lambda_1}$ is evident, while the one for $\mathcal{E}_{\Lambda_2}$ is found as follows. First, one computes the Schur element $c_{{\lambda_2}^\prime}$ of the conjugate partition ${\lambda_2}^\prime$
		of $\lambda_2$ by using the Steinberg formula (see, e.g.,~\cite[Sect. 2]{JLM0} with $q=e^{h/2}$), thus getting $c_{{\lambda_2}^\prime}=1+e^h$. Next, we observe that
		\[
			\R_{\Lambda_2}(z_1, z_2) \Big|_{z_1 =1}\Big|_{z_2 =1}
			=P\ts R -e^{h/2} \ts I.
		\]
		Indeed, this follows immediately from~\eqref{givenproduct}, as the given product contains only one factor for $n=2$. Finally, the identity $\check{R}_0^{-1}=R^{-1}P$ is recovered from~\cite[Sect. 2, 3]{JLM0}.

		Therefore, the expression~\eqref{yuhgln1} for $\Lambda_1$ and $\Lambda_2$ yields
		\begin{equation}
		\tr
		\left(z + h\ts L^+ \left(L^-\right)^{-1}\right)
		D
		\end{equation} and
		\begin{equation}
		\textstyle\frac{1}{1+e^h}
		\tr_{1,2}
		\left(zL_1^- + h L_1^+\right)
		\left(zL_2^- + he^{-h} L_2^+ \right)
		\left(L_2^-\right)^{-1}
		\left(L_1^-\right)^{-1}
		D^{\otimes 2} \left(I-e^{h/2}R^{-1}P\right).
		\end{equation}
	\end{example}

	\begin{remark}
		The Etingof--Kazhdan construction~\cite{EK} of the quantum affine vertex algebra of type $A$, which we recalled in Theorem~\ref{EK:qva}, naturally extends to the types $B$, $C$ and $D$, i.e., to the orthogonal and symplectic case~\cite{Kbcd}. However, due to the lack of the standard quantum evaluation morphisms outside of type $A$, it is not clear how to generalize the results of this paper to other types. The right framework might be established by using the Kirillov--Reshetikhin modules~\cite{KirRes}.
	\end{remark}

	\section{Deformed \texorpdfstring{$\Vc$}{Vc(glN)}-modules associated with \texorpdfstring{$U(\gl_N)$}{U(glN)}}\label{section_0 6}

	In this section, we use our recent results~\cite{BK0} to partially extend the constructions from Sections~\ref{compatible} and~\ref{section_0 5} to the case of the normalized \emph{Yang $R$-matrix},
	\begin{equation}\label{Rrat}
	R(u) = g(u/h) \left(1-hu^{-1}P\right) \in\ndo\CC^N\otimes\ndo\CC^N [[h/u]],
	\end{equation} where $P$ is the permutation operator~\eqref{permutation} and
	$g(u)$ a~unique element of $ 1+ u^{-1}\CC[[ u^{-1}]]$ satisfying
	$g(u+N)=g(u)(1-u^{-2})$. From now on, we consider only the rational setting, governed by~\eqref{Rrat}, so the fact that some objects are denoted by the same symbols as their trigonometric counterparts in previous sections should not cause any confusion.

	The Etingof--Kazhdan construction~\cite[Thm. 2.3]{EK} of the quantum affine vertex algebra associated with~\eqref{Rrat} employs the structure of the \emph{double Yangian $\dy$ for $\mathfrak{gl}_N$ at the level $c\in\CC$}~\cite{I}. It is defined as the $h$-adically completed associative algebra over the ring $\CC[[h]]$ generated by the elements
	$t_{ij}^{(\pm r)}$, where $i,j=1,\ldots,N$ and $r=1,2,\ldots,$ subject to the defining relations
	\begin{align*}
		R(u-v)\ts T_1^{\pm}(u)\ts T_2^{\pm}(v)
		=&\,\, T_2^{\pm}(v)\ts T_1^{\pm}(u)\ts R(u-v),\\
		R(u-v+hc/2)\ts T_1^{-}(u)\ts T_2^{+}(v)
		=&\,\, T_2^{+}(v)\ts T_1^{-}(u)\ts R(u-v-hc/2).
	\end{align*}
	The generator matrices $T^{\pm}(u)\in\ndo\CC^N\otimes \dy[[u^{\pm 1}]]$ are given by
	\[
		T^{\pm}(u)
		=\sum_{i,j=1}^N e_{ij}\otimes t_{ij}^{\pm (u)},\quad\text{where}\quad t_{ij}^+ (u)=\delta_{ij}-h\sum_{r\geqslant 1} t_{ij}^{(-r)}\ts u^{r-1},\,\, t_{ij}^- (u)=\delta_{ij}+h\sum_{r\geqslant 1} t_{ij}^{(r)}\ts u^{-r}.
	\]
	The corresponding quantum vertex algebra structure~\cite[Thm. 2.3]{EK} is defined over the \emph{ vacuum module $\Vc$} for the double Yangian, which is the quotient of the algebra $\dy$ over its left ideal generated by all $t_{ij}^{(r)}$ with $i,j=1,\ldots,N$ and $r=1,2,\ldots.$

	\begin{theorem}
		Let $c\in\CC.$ There exists a~unique structure of quantum vertex algebra over $\Vc $
		such that the vertex operator map $Y(\cdot,z)$ is given by
		\[
			Y\big(T^+_{[n]} (u)\vac,z\big)=T^+_{[n]} (z|u)\ts T_{[n]}^- (z+hc/2|u)^{-1},
		\]
		the vacuum vector is $\vac$
		and the braiding map $\mathcal{S}$ is defined by
		\begin{align*}
			&\mathcal{S}(z)\left(R_{nm}^{ 12}(z+u-v)^{-1}\ts T_{[m]}^{+ 24}(v) \ts R_{nm}^{ 12}(z+u-v -h c) \ts T_{[n]}^{+ 13}(u) (\vac\otimes \vac) \right)\nonumber\\
			=\, & T_{[n]}^{+ 13}(u)\ts R_{nm}^{ 12}(z+u-v +h c)^{-1} \ts
			T_{[m]}^{+ 24}(v) \ts R_{nm}^{ 12}(z+u-v)(\vac\otimes \vac).
		\end{align*}
	\end{theorem}

	The operators $T^\pm_{[n]}$ over the vacuum module $\Vc$ are defined in parallel with~\eqref{tplusovi_t rig} and the $R$-matrix products as in~\eqref{Rnm1}, so that we have
	\[
		R_{nm}^{12}(z+u -v +ah)=
		\prod_{r=1,\dots,n }^{\longrightarrow} \prod_{s=n+1,\dots,n+m }^{\longleftarrow}
		R_{rs}(z+ u_r -v_{s-n} + ah).
	\]
	The next two definitions are based on~\cite[Def. 3.1, 3.3]{BK0}. They can be regarded as additive versions of Definitions~\ref{def_b r_m ult} and~\ref{weaklyassoc} above. However, note that some requirements imposed on the corresponding maps $\sigma$ and $\rho$ are strengthened.

	\begin{definition}
		Suppose $\mathcal{V}$ is a~topologically free $\CC[[h]]$-module and
		\[
			\sigma(z),\rho(z)\colon \mathcal{V}\otimes \mathcal{V}\to \mathcal{V}\otimes \mathcal{V}\otimes \CC [[z^{-1},h]]
		\]
		are $\CC[[h]]$-module maps. The pair $(\sigma, \rho)$ is said to be an \emph{ additive compatible pair} if it satisfies the following conditions.
		\begin{enumerate}
			\item The map $\sigma$ satisfies the quantum Yang--Baxter equation,
			\begin{align}
				&\sigma_{12}(z_1)\ts\sigma_{13}(z_1 +z_2)\ts\sigma_{23}(z_2)=
				\sigma_{23}(z_2)\ts\sigma_{13}(z_1 +z_2)\ts\sigma_{12}(z_1) \nonumber \\
				\intertext{and the unitarity condition,}
				&\sigma(-z)\ts \sigma_{21}(z)=\sigma_{21}(z)\ts\sigma(-z) =1. \label{uni_r ev}
			\end{align}
			\item The map $\rho(z)$ is invertible.

			\item The map
			\begin{align}
				\mathcal{S}_{\sigma,\rho}(z) \colon \mathcal{V}\otimes \mathcal{V} &\to \mathcal{V}\otimes \mathcal{V},\nonumber\\
				\mathcal{S}_{\sigma,\rho}(z) &= \rho (z)\ts \sigma (z)\ts \rho_{21}(-z)^{-1} \big.\label{Ssigmarhorat}
			\end{align}
			satisfies the quantum Yang--Baxter equation,
			\[
				\mathcal{S}_{\sigma,\rho}(z_1)_{12} \ts
				\mathcal{S}_{\sigma,\rho}(z_1 +z_2)_{13} \ts
				\mathcal{S}_{\sigma,\rho}(z_2)_{23}
				=
				\mathcal{S}_{\sigma,\rho}(z_2)_{23} \ts
				\mathcal{S}_{\sigma,\rho}(z_1 +z_2)_{13} \ts
				\mathcal{S}_{\sigma,\rho}(z_1)_{12}.
			\]
		\end{enumerate}
	\end{definition}

	\begin{remark}
		Note that the unitarity~\eqref{uni_r ev} of the map $\sigma $ implies the unitarity of $\mathcal{S}_{\sigma,\rho} $,
		\[
			\mathcal{S}_{\sigma,\rho}(z)\ts\mathcal{S}_{\sigma,\rho}(-z)_{21}=\mathcal{S}_{\sigma,\rho}(-z)_{21}\ts\mathcal{S}_{\sigma,\rho}(z) =1.
		\]
	\end{remark}

	\begin{definition}
		Let $(\mathcal{V},Y,\Sc,\vac)$ be an $h$-adic quantum vertex algebra and $(\sigma,\rho)$ an additive compatible pair over $\mathcal{V}$. The pair $(\sigma,\rho)$ is said to be \emph{ associated} with the $h$-adic quantum vertex algebra $\mathcal{V}$ if the map $\mathcal{S}_{\sigma,\rho}(z)$, as given by~\eqref{Ssigmarhorat}, satisfies $\mathcal{S}_{\sigma,\rho}(z)=\mathcal{S}(z)$.
	\end{definition}

	Let us recall the compatible pair from~\cite[Sect. 4]{BK0}.

	\begin{lemma}\label{lem_s igma}
		Let $c\in\CC$. There exist unique maps
		\[
			\sigma(z),\rho(z)\colon \Vc\otimes\Vc\to\Vc\otimes\Vc\otimes\CC[[h]][z^{-1}]_h
		\]
		such that on
		$
		(\ndo\mathbb{C}^{N})^{\otimes n} \otimes
		(\ndo\mathbb{C}^{N})^{\otimes m}\otimes \Vc \otimes \Vc$
		we have
		\begin{align*}
			&\sigma(z)\Big(T_{[n]}^{+13}(u) \ts T_{[m]}^{+24}(v) (\vac\otimes\vac)\Big) \\
			= &\, R_{nm}^{ 12}(z+u-v)\ts T_{[n]}^{+13}(u) \ts T_{[m]}^{+24}(v)
			R_{nm}^{ 12}(z+u-v)^{-1}(\vac\otimes\vac),\\
			&\rho(z)\Big(T_{[n]}^{+13}(u) \ts T_{[m]}^{+24}(v)(\vac\otimes\vac) \Big) \\
			= & \,T_{[n]}^{+13}(u) \ts R_{nm}^{ 12}(z+u-v+h c)^{-1}\ts T_{[m]}^{+24}(v)
			R_{nm}^{ 12}(z+u-v) (\vac\otimes\vac).
		\end{align*}
		Moreover, the maps $\sigma$ and $\rho$ form an additive compatible pair associated with $\Vc$.
	\end{lemma}

	\begin{remark}
		Suppose $(\mathcal{V},Y,\Sc,\vac)$ is an $h$-adic quantum vertex algebra and $(\sigma,\rho)$ an additive compatible pair over $\mathcal{V}$. In parallel with~\eqref{Yrhho}, one can associate with the vertex operator map $Y$ the map $Y^\rho$ by
		\[
			Y^\rho(z):= Y(z)\ts \rho(z)
			\colon \mathcal{V}\otimes\mathcal{V} \to\mathcal{V}\otimes \CC((z))[[h]].
		\]
		Furthermore, by~\cite[Eq. 4.5]{BK0}, we have
		\[
			Y^\rho(T_{[n]}^{+13}(u)\vac,z)T_{[m]}^{+24}(v)\vac
			=T_{[n]}^{+13}(z|u)\ts T_{[m]}^{+24}(v)(\vac\otimes \vac),
		\]
		so that the map $\rho$ possesses the same property of cancelling the singularities of the vertex operator map, as its trigonometric counterpart; recall Remark~\ref{Yrho_r em}.
	\end{remark}

	The next definition is a~minor adjustment of~\cite[Def. 3.5]{BK0} to the $h$-adic setting.

	\begin{definition}
		Let $(\mathcal{V},Y,\vac,\Sc)$ be an $h$-adic quantum vertex algebra and $(\sigma,\rho)$ an additive compatible pair associated with $\mathcal{V}$. Let $W$ be a~topologically free $\CC[[h]]$-module equipped with a~$\CC[[h]]$-module map
		\begin{align*}
			Y_W (\cdot,z) \colon \mathcal{V}\otimes W&\to W((z))_h, \\
			v\otimes w&\mapsto Y_W (z)(v\otimes w)=Y_W (v,z)w=\sum_{r\in\mathbb{Z}} v_r w \ts z^{-r-1}.\nonumber
		\end{align*}
		A pair $(W,Y_W)$ is said to be a~\emph{$(\sigma,\rho)$-deformed $\mathcal{V}$-module} if the map $Y_W (\cdot, z)$ possesses the following properties:
		\begin{align}
			&Y_W (\vac, z)w=w\quad\text{for all }w\in W,\nonumber \\
			\intertext{the \emph{weak $\rho$-associativity}:
			for any $u,v\in \mathcal{V}$, $w\in W$ and $n\in\ZZ_{>0}$ there exists $r\in\ZZ_{\geqslant 0}$ such that}
			& (z_0 +z_2)^r\ts Y_W (u,z_0 +z_2)\ts Y_W (v,z_2)\ts w \nonumber\\
			&\qquad - (z_0 +z_2)^r\ts Y_W\big(Y^\rho(u,z_0)v,z_2\big)\ts w \in h^n\ts W[[z_0^{\pm 1}, z_2^{\pm 1}]], \label{assoc_r ho_r at}
			\intertext{and the \emph{$\sigma$-locality}: for any $u,v\in \mathcal{V}$ and $n\in\ZZ_{>0}$ there exists $r\in\ZZ_{\geqslant 0}$ such that }
			& (z_1 -z_2)^r \ts Y_W (z_1)\big(1\otimes Y_W (z_2)\big)\big(\sigma(z_1 -z_2)(u\otimes v)\otimes w\big)\nonumber\\
			&\qquad - (z_1 -z_2)^r \ts Y_W (v,z_2) Y_W (u,z_1) w
			\in h^n\ts W[[z_1^{\pm 1}, z_2^{\pm 1}]] \qquad \text{for all }w\in W.\label{sigma_l oc_t op_r at}
		\end{align}
	\end{definition}

	The \emph{dual Yangian $Y^+ (\mathfrak{gl}_N)$ for $\mathfrak{gl}_N$} is defined as an $h$-adically completed subalgebra of the double Yangian $DY_c (\mathfrak{gl}_N)$ generated by all $t_{ij}^{(-r) }$ with $r\geqslant 1$. As a~$\CC[[h]]$-module, $Y^+ (\mathfrak{gl}_N)$ coincides with $\Vc$. Consider the universal enveloping algebra $U(\mathfrak{gl}_N)$. Let $E=(E_{ij})$ be the matrix whose $(i,j)$-entry is the element $E_{ij}\in U(\mathfrak{gl}_N)$, i.e.,
	\[
		E=\sum_{i,j=1,\ldots,N} e_{ij}\otimes E_{ij}\in\ndo\CC^N \otimes U(\mathfrak{gl}_N).
	\]
	The following construction of deformed modules, given in Theorem~\ref{ug} below, is motivated by the fact that for any nonzero $a\in \CC$ the mapping
	\[
		T^+ (u)\mapsto 1+ E h(a+u)^{-1}
	\]
	defines the \emph{evaluation homomorphism}
	$Y^+ (\mathfrak{gl}_N)\to U(\mathfrak{gl}_N)[[h]] $. Let
	\[
		E(z)=1+ E h z^{-1}\in\ndo\CC^N\otimes U(\mathfrak{gl}_N)[[h]] [z^{-1}].
	\]
	For any $n\geqslant 1$ and the family of variables $u=(u_1,\ldots,u_n)$, define the operator
	\[
		E_{[n]}(z+u) =E_1 (z+u_1)\ldots E_n (z+u_n),
	\]
	where its action is given by the algebra multiplication. Clearly, $E_{[n]}(z+u)$ belongs to $\left(\ndo\CC^N\right)^{\otimes n}\otimes\ndo U(\mathfrak{gl}_N)[[h]] [z^{-1}][[u_1,\ldots,u_n ]]$. The following theorem, which employs the additive compatible pair from Lemma~\ref{lem_s igma}, can be regarded as a~rational counterpart of Corollary~\ref{kor_4 4}.

	\begin{theorem}\label{ug}
		Let $c\in\CC $. Suppose $W$ is a~$ U (\gl_N)$-module. There exists a~unique structure of $(\sigma,\rho)$-deformed $\Vc$-module over $W[[h]]$ such that the corresponding module map $Y_{W[[h]]}(\cdot,z)$ satisfies
		\begin{equation}\label{fothemp}
		Y_{W[[h]]}(T_{[n]}^+ (u)\vac,z) = E_{[n]}(z+u)_{W[[h]]}.
		\end{equation}
	\end{theorem}

	\begin{proof}
		The theorem can be verified by closely following the proof of~\cite[Thm. 5.1]{BK0}. More precisely, the setting of~\cite[Thm. 5.1]{BK0} differs from the present one in inessential respects. In particular, a~different topology is considered, but these differences do not affect the argument. The vertex operators~\eqref{fothemp} and the vertex operators from~\cite[Thm. 5.1]{BK0} satisfy the same FRT-relations and both settings employ the same compatible pair $(\sigma, \rho)$, so the original proof can be copied almost verbatim. Furthermore, both the weak $\rho$-associativity~\eqref{assoc_r ho_r at} and the $\sigma$-locality~\eqref{sigma_l oc_t op_r at} can be shown to hold for $r=0$.
	\end{proof}
	At the end, we discuss a~relation between the quantum Feigin--Frenkel center and the central elements in the universal enveloping algebra from the viewpoint of $(\sigma,\rho)$-deformed modules. As before, let $\Lambda$ be a~standard tableau of shape $\lambda\vdash n$ with at most $N$ rows and $c_k (\Lambda)$ the content $j-i$ of the box $(i,j)$ of $\lambda$ occupied by $k$ in $\Lambda $. Let
	\[
		\R(u_1,\ldots,u_n)
		=
		\prod_{1\leqslant i<j\leqslant n }
		\R_{ij}(u_i -u_j) \quad\text{for}\quad
		\R(u)=1-hu^{-1}P,
	\]
	where the product is taken in the lexicographical order on the set of pairs $(i,j)$. By the fusion procedure originated in~\cite{j:yo},
	\begin{equation}\label{fusion_r ev}
	\mathcal{E}_\Lambda:=\frac{1}{h(\lambda)}\R(u_1,\ldots,u_n) \Big|_{u_1 =hc_1 (\Lambda)}\Big|_{u_2 =hc_2 (\Lambda)}\ldots \Big|_{u_n =hc_n (\Lambda)},
	\end{equation} where $h(\lambda)$ is the product of all hook lengths of the boxes of $\lambda$, is a~well-defined operator satisfying $\mathcal{E}_\Lambda^2 =\mathcal{E}_\Lambda$;
	see also~\cite[Sect. 6.4]{m:yc} for more details and references.

	Denote by $\Vn$ the quantum vertex algebra $\Vc$ at the critical level $c=-N$. Consider the formal power series
	\begin{equation}\label{teelambda}
	T_\Lambda(u)=\tr_{1,\ldots,n} \, T_{[n]}^+ (u^{(\Lambda)})\vac \ts \Ec_\Lambda \in \Vn[[u]],
	\end{equation} where $u^{(\Lambda)}=(u+hc_1 (\Lambda),\ldots,u+hc_n (\Lambda))$\footnote{Note that, throughout this section, we use the notation $u^{(\Lambda)}=(u+hc_1 (\Lambda),\ldots,u+hc_n (\Lambda))$, which differs from the expression for $u^{(\Lambda)}$ used in the previous section; recall~\eqref{yuhgln}.}
	and the trace is taken over the tensor factors $1,\ldots,n$. Due to~\cite[Thm. 2.4]{JKMY}, all coefficients of $T_\Lambda(u)$ belong to the center of the quantum vertex algebra $\Vn$.

	Set $W= U (\gl_N)$ in Theorem~\ref{ug} and then consider the image of the constant term
	$T_\Lambda(0)$ of the series~\eqref{teelambda} under the map
	$Y_{ U (\gl_N)[[h]]}(\cdot,zh)$. Multiplying the image by the polynomial $(z+ c_1 (\Lambda))\ldots (z+ c_n (\Lambda))$, one obtains
	\[
		\mathbb{S}_\lambda (z):=
		\tr_{1,\ldots,n}
		\left(z+ c_1 (\Lambda)+E_1 \right)\ldots \left(z+ c_n (\Lambda)+E_n \right)
		\Ec_\Lambda \in U (\gl_N)[z ].
	\]
	Finally, we remark that the elements $\mathbb{S}_\lambda =\mathbb{S}_\lambda (0)$
	given by
	\[
		\mathbb{S}_\lambda =
		\tr_{1,\ldots,n}
		\left(c_1 (\Lambda)+E_1 \right)\ldots \left(c_n (\Lambda)+E_n \right)
		\Ec_\Lambda \in U (\gl_N),
	\]
	originated in~\cite{Ok}, are called \emph{quantum immanants}. The polynomials $\mathbb{S}_\lambda (z)$ are independent of the choice of the standard tableau $\Lambda$ of shape $\lambda$. Furthermore, all
	$\mathbb{S}_\lambda$ with $\lambda$ running over all diagrams with at most $N$ rows form a~basis of the center of the universal enveloping algebra $U(\mathfrak{gl}_N)$. For more information on quantum immanants see also~\cite[Sect. 4.9]{Mnew}.

	\begin{example}
		This example illustrates the connection between the central elements of $\mathcal{V}^{crit} (\mathfrak{gl}_2)$ and those of $U(\mathfrak{gl}_2)$. Consider the standard tableaux~\eqref{statableaux}
		of shapes
		$\lambda_1 =(1)$ and $\lambda_2 =(1,1)$. Their contents are given by~\eqref{contents} and the products of hook lengths by~\eqref{hooks}. The corresponding idempotents, constructed by the fusion procedure~\eqref{fusion_r ev}, are
		\[
			\mathcal{E}_{\Lambda_1} =I\in\ndo \CC^2\Fand
			\mathcal{E}_{\Lambda_2} =A=\textstyle\frac{1}{2} \left(I-P\right)\in\ndo \CC^2\otimes \ndo \CC^2,
		\]
		where $I$ denotes the identity (on $\CC^2$ and also on $\CC^2\otimes\CC^2$) and $P$ the permutation operator defined by~\eqref{permutation} for $N=2$. Clearly, $\mathcal{E}_{\Lambda_2}$ equals the normalized antisymmetrizer $A$\footnote{Note that the normalized antisymmetrizer $A$ can be also obtained from the expression for $\mathcal{E}_{\Lambda_2}$ in~\eqref{elamba12trig} by setting $h=0$.}. The power series~\eqref{teelambda}, whose coefficients belong to the center of $\Vndva$, are given by
		\[
			T_{\Lambda_1}(u)=\tr\, T^+ (u)\vac\Fand
			T_{\Lambda_2}(u)=\tr_{1,2}\, T_1^+ (u)T_2^+ (u-h)\vac\ts A.
		\]
		Note that $T_{\Lambda_2}(u)$ is exactly the quantum determinant of the matrix $T^+ (u)$. By using~\eqref{fothemp}, we find that the images of their constant terms $T_{\Lambda_1}(0)$ and $T_{\Lambda_2}(0)$ under the map $Y_{ U (\gl_2)[[h]]}(\cdot,zh)$ are equal to
		\[
			\tr \left(1+Ez^{-1}\right) \Fand
			\tr_{1,2} \left(1+E_1 z^{-1}\right)\left(1+E_2 (z-1)^{-1}\right) A.
		\]
		Finally, by multiplying the first (resp. second) expression by $z+c_1 (\Lambda_1)=z$ (resp. $(z+c_1 (\Lambda_2))(z+c_2 (\Lambda_2))=z(z-1)$), we obtain the Okounkov quantum immanants for $U (\gl_2)$,
		\[
			\mathbb{S}_{\lambda_1} (z)= \tr \left(z+E \right) \Fand
			\mathbb{S}_{\lambda_2} (z)=\tr_{1,2} \left(z+E_1 \right)\left(z-1+E_2 \right) A.
		\]
	\end{example}

        
\end{document}